\documentclass[12pt]{amsart}
\usepackage{latexsym, graphicx, hyperref, amsmath}
\usepackage[vcentermath]{youngtab}
\usepackage[margin=1in]{geometry}
\usepackage{amssymb}
\usepackage{enumitem}
\usepackage{fancyhdr}
\usepackage{tikz}
\usepackage{comment}
\usepackage{tikz-cd}
\usepackage{tikz} 
\usepackage{url}
\usepackage[utf8]{inputenc}
\usepackage[
backend=biber,
style=alphabetic,
sorting=ynt
]{biblatex}
\usepackage{ytableau}
\usepackage{array}
\usetikzlibrary{positioning}
\allowdisplaybreaks

\bibliography{refs.bib}

\usetikzlibrary{trees}

\newcommand{\C}{{\mathbb C}}

\newcommand{\Z}{{\mathbb Z}}

\newcommand{\sbar}{\overline{s}}

\newtheorem{theorem}{Theorem}[section]
\newtheorem{proposition}[theorem]{Proposition}
\newtheorem{observation}[theorem]{Observation}
\newtheorem{lemma}[theorem]{Lemma}
\newtheorem{Claim}[theorem]{Claim}
\newtheorem{claim}[theorem]{Claim}

\newtheorem{corollary}[theorem]{Corollary}
\newtheorem{Main Conjecture}[theorem]{Main Conjecture}

\newtheorem{definition}[theorem]{Definition}
\newtheorem{example}[theorem]{Example}

\newtheorem{problem}{Problem}

\title{Combinatorics and Representation Theory of Special Cases of Chern Plethysm}
\author{Nathaniel Libman \and Gidon Orelowitz}

\begin{document}

\begin{abstract}
    Chern plethysm (introduced by Billey, Rhoades, and Tewari) is a geometric way to produce Schur-positive symmetric polynomials. We present combinatorial interpretations for the Schur expansions of special cases of Chern plethysm. We also exhibit a symmetric group module whose Frobenius characteristic is (a symmetric function analog of) one of these cases, generalizing a result of Reiner and Webb.
\end{abstract}

\maketitle

\section{Introduction}
Schur polynomials, and more generally Schur positive symmetric polynomials, are ubiquitous in algebraic combinatorics and representation theory. The Schur polynomials are the characters of the irreducible polynomial representations of the general linear group, and they are explicitly defined in terms of \emph{semistandard Young tableaux} (precise definitions will be recalled in Section 2). They form a basis of the complex vector space of symmetric polynomials in $n$ variables. A symmetric polynomial is \emph{Schur positive} if it can be written as a nonnegative integer linear combination of Schur polynomials.

A Schur polynomial is determined by a partition $\lambda = (\lambda_1, \lambda_2, \dots, \lambda_l)$ and a positive integer $n$ (the number of variables). When $\lambda = (m)$, the corresponding Schur polynomial $s_\lambda$ is called the \emph{homogeneous symmetric polynomial}, and when $\lambda = (1^m) := (1, 1, \dots, 1)$, the corresponding Schur polynomial is called the \emph{elementary symmetric polynomial}. In these cases, we say that $\lambda$ is a \emph{single row} or \emph{single column}, respectively. For example, if $n=3$ and $\lambda = (1,1)$, the corresponding Schur polynomial is
\[
s_{(1,1)}(x_1, x_2, x_3) = x_1x_2 + x_2x_3 + x_3x_1
\]

Chern plethysm is a geometric way to produce Schur positive symmetric polynomials. It is related to the notion of \emph{classical plethysm} of Schur polynomials; it involves a composition operation which appears simple in principle, but produces very complex results. The Schur positivity of Chern plethysm follows from work of Pragacz (\cite{pragacz}) and Fulton-Lazarsfeld (\cite{fultonLaz}). In \cite{billeyrhoadestewari}, Billey-Rhoades-Tewari explored a particular case of Chern plethysm: the \emph{Boolean product polynomials}, defined as follows.
\begin{definition}
    For $1 \leq k \leq n$, let
    \[
    B_{n,k} := \prod_{1 \leq i_1 < i_2 < \dots < i_k \leq n} (x_{i_1} + \dots + x_{i_k})
    \]
\end{definition}
Of particular interest in that paper is the case $k=n-1$. In that case, the work by Reiner and Webb in \cite{reinerwebb} gives a combinatorial interpretation for the Schur coefficients in terms of descents of \emph{standard Young tableaux}. That work also gives a representation-theoretic interpretation of this case in terms of poset homology.

Expanding on \cite{reinerwebb}, in this paper we consider a more general class of Chern plethysms for which the Boolean product polynomial with $k=n-1$ is just one example. In general, the result of Chern plethysm depends on two partitions (an \emph{inner} partition $\mu$ and an \emph{outer} partition $\lambda$) and $n$, the number of variables involved. This paper concerns the case $\mu = (1^{n-1})$. We can think of the resulting Chern plethysm as plugging in the sums of all but one of the variables at a time into a Schur polynomial. For example, if $\lambda = (1^n)$, we recover the Boolean product polynomial $B_{n, n-1}$.

Our first results (Theorems \ref{thm:columncase} and \ref{thm:rowcase}) give combinatorial expansions for this case of Chern plethysm when $\lambda$ is a single column or single row. Both results give tableaux-theoretic interpretations for the coefficients: in the column case, the coefficients count certain standard tableaux we call \emph{parity tableaux}; in the row case the coefficients count standard tableaux of a given modified shape.

We then turn to finding combinatorial interpretations for the coefficients of specific terms in the Schur expansions of Chern plethysm with $\mu = (1^{n-1})$ and general $\lambda$. We are able to find compact combinatorial descriptions for the coefficients indexed by one row, two rows, or a column (\ref{cor:singlerow}, \ref{cor:tworows}, and \ref{thm:singlecolumn}). The coefficient indexed by one row is simply counted by a set of semistandard tableaux, while the coefficients indexed by two rows are counted by a certain sum of pairs of skew tableaux with conditions. The computation of the coefficient of a column uses the well-known Lindstr\"om-Gessel-Viennot lemma for counting lattice paths (\cite{gessel}), and gives a formula in terms of fillings of Young diagrams which are not quite semistandard. Finding a complete combinatorial description of all of the coefficients for general $\lambda$ remains open.

We conclude with Theorem \ref{thm:c(x)}, which gives a representation-theoretic interpretation of the case when $\lambda$ is a column. We are no longer able to use homology of posets (as in \cite{reinerwebb}), but we are able to use spectral sequences and a result in \cite{gan} to relate Chern plethysm to the homology of a certain algebraic complex. 

\section{Definitions and preliminary results}

A \emph{partition} $\lambda = (\lambda_1, \lambda_2, \dots, \lambda_l)$ is a weakly decreasing sequence of nonnegative integers. A \emph{Young diagram} for the partition $\lambda$ is an upper left justified series of boxes with $\lambda_i$ boxes in each row. Given two partitions $\lambda$ and $\mu$ with $\mu_i \leq \lambda_i$ for all $i$, the skew shape $\lambda / \mu$ is the Young diagram for $\lambda$ with the Young diagram for $\mu$ deleted from it.

For integers $a\leq b$, we let $[a,b] := \{a,a+1,\dots, b-1,b\}$, and we write $[1,n]$ as $[n]$.
For a set of integers $S$ and a skew partition $\lambda/\mu$, we define the corresponding set of semistandard Young tableaux, denoted ${\rm SSYT}(\lambda/\mu,S)$, to be the set of all fillings of the skew shape $\lambda / \mu$ with the numbers $S$ such that the entries of each filling are weakly increasing along rows and strictly increasing along columns.  When $S = [n]$, the brackets are omitted, and when $S = \mathbb{N}$ the argument is omitted entirely.  Given $T\in {\rm SSYT}(\lambda/\mu,S)$ and integers $r,c$ such that $\mu_r < c\leq \lambda_r$, let $T(r,c)$ denote the label in row $r$ and column $c$ of $T$.

We define ${\rm SYT}(\lambda/\mu,n)\subseteq {\rm SSYT}(\lambda/\mu,n)$ to be the elements of ${\rm SSYT}(\lambda/\mu,n)$ that use each value at most once and are strictly increasing along rows. Note that we allow entries larger than the number of boxes in this definition. Let $f^{\lambda/\mu,n} = |{\rm SYT}(\lambda/\mu,n)|$.  When $\mu = (0)$ it is omitted, and when $n = |\lambda| - |\mu|$, $n$ is occasionally omitted from the argument of ${\rm SYT}$ and $f^{\lambda/\mu,n}$ as well. Finally, when $\alpha$ is not a skew partition, we define $f^\alpha$ to be zero.

Given an element $T \in {\rm SSYT}(\lambda, n)$, we define $wt(T) = (wt_1(T),\dots, wt_n(T))$, where $wt_i(T)$ is the number of $i$'s that appear in $T$.
For a tuple $a = (a_1,\dots, a_n)$, define $x^a:= x_1^{a_1}\dots x_n^{a_n}$.  
The \emph{Schur polynomial} indexed by $\lambda$ is defined as follows.

\begin{definition}
    \[
    s_\lambda(x_1, \dots, x_n) := \sum_{T \in {\rm SSYT}(\lambda, n)} x^{wt(T)}
    \]
\end{definition}
Let $\varepsilon := \sum_{w \in \mathfrak{S}_n} \text{sign}(w) \cdot w$ denote the antisymmetrizing element of the group algebra $\mathbb{C}[\mathfrak{S}_n]$ of the symmetric group on $n$ letters. The group $\mathfrak{S}_n$ acts on the polynomial ring $\mathbb{C}[x_1, \dots, x_n]$ by permuting variable indices and the group algebra acts by linear extension. We have the following useful formula for the Schur polynomial $s_\lambda(x_1, \dots, x_n)$. (This is the \emph{bialternant formula}, which is actually the original definition dating back to 1841; see \cite{Jacobi1841}).
        \begin{equation}
        \label{eq:bialt}
            s_\lambda(x_1, \dots, x_n) = \frac{\varepsilon \cdot x^{\lambda + \delta}}{\varepsilon \cdot x^\delta}
        \end{equation}
where $\delta = (n-1, n-2, \dots, 1, 0)$.

Let $X = \mathbb{P}^\infty \times \dots \times \mathbb{P}^\infty$ be the $n$-fold product of infinite-dimensional complex projective space, and let $\mathcal{E} = \ell_1 \oplus \dots \oplus \ell_n$ be the direct sum of the tautological line bundles over the $n$ factors of $X$. Then we have the presentation $H^\bullet(X) = \Z[x_1, \dots, x_n]$, where the Chern roots of $\mathcal{E}$ are the variables $x_1, \dots, x_n$.

Let $V = \mathbb{C}^n$ be the standard $n$-dimensional complex vector space. We refer the reader to Chapter $8$ of \cite{fulton_1996} for the definition of the \emph{Schur functor} $\mathbb{S}^\mu(V)$ (therein referred to as a \emph{Schur module} $V^\mu$). We also define $\mathbb{S}^\mu(\mathcal{E})$ to be the new bundle with fibers $\mathbb{S}^\mu(\mathcal{E})_p := \mathbb{S}^\mu(\mathcal{E}_p)$. We note that the Chern roots of $\mathbb{S}^\mu(\mathcal{E})$ are the sums $\sum_{\square \in T} x_\square$, where $T$ ranges over all elements of $\text{SSYT}(\mu, n)$.

In \cite{billeyrhoadestewari}, the authors defined the notion of \emph{Chern plethysm} $F(\mathcal{E})$, where $F$ is a symmetric function. In that paper, and in \cite{pragacz} (using work of Fulton-Lazarsfeld in \cite{fultonLaz}) the authors showed that if $\lambda$ and $\mu$ are partitions, then $s_\lambda(\mathbb{S}^{\mu}(\mathcal{E}))$ is always Schur positive. We restate the definition of Chern plethysm in this case here:
\begin{definition}
    Given two partitions $\lambda$ and $\mu$, we have 
    \[
    s_\lambda(\mathbb{S}^{\mu}(\mathcal{E})) = s_\lambda\left(\dots, \sum_{\square \in T} x_\square, \dots\right)
    \]
    where the entries of $s_\lambda$ are indexed by all semistandard Young tableaux $T$ of shape $\mu$, and $\square \in T$ refers to the multiset of labels in the tableau $T$ (we implicitly set all other entries of $s_\lambda$ to $0$). 
\end{definition}

The problem of determining combinatorial or representation-theoretic interpretations for this Schur positivity remains open in general.

In this paper, we consider the symmetric polynomials $s_\lambda(\mathbb{S}^\mu(\mathcal{E}))$ in the special case where $\mu = (1^{n-1})$. Fixing $\mu = (1^{n-1})$, the Chern roots of $\mathbb{S}^{\mu}(\mathcal{E})$ are $\{(x_1+\dots+x_n)-x_i: 1 \leq i \leq n\}$. We will denote $s_{\lambda}(\mathbb{S}^{\mu}(\mathcal{E}))$ by
\[
\sbar_\lambda(x_1,\dots, x_n) := s_{\lambda}((x_1+\dots+x_n)-x_1, \dots, (x_1+\dots+x_n)-x_n)
\]
For example, 
\begin{align*}
    \sbar_{(1)}(x_1,\dots, x_n) & = \sum_{i=1}^n (x_1+\dots+x_n)-x_i \\
    & = n(x_1+\dots+x_n) - (x_1+\dots+x_n)\\
    & = (n-1)(x_1+\dots+x_n) \\
    & = (n-1)s_{(1)}(x_1,\dots, x_n)
\end{align*}

We also use the notation
\[
[s_\mu]\sbar_\lambda(x_1,\dots x_n)
\]
to denote the coefficient of $s_\mu$ in the Schur expansion of $\sbar_\lambda(x_1,\dots, x_n)$. We begin with a simple observation.

\begin{observation}
\label{prop:degree}
$[s_\mu]\sbar_\lambda(x_1,\dots x_n) = 0$ unless $|\lambda| = |\mu|$
\end{observation}
\begin{proof}
We have that $\sbar_\lambda(x_1,\dots, x_n)$ and $s_\mu(x_1,\dots, x_n)$ are homogeneous polynomials of degree $|\lambda|$ and $|\mu|$, respectively.
\end{proof}

The following proposition will prove useful in later calculations:

\begin{proposition}
\label{prop:boxmult}
\[s_{\mu}(x_1,\dots, x_n) s_{(1)}^k(x_1,\dots, x_n) = \sum_{|\lambda| = |\mu| + k} f^{\lambda/ \mu}s_{\lambda}(x_1,\dots, x_n)\]
\end{proposition}
\begin{proof}
This follows from repeated application of Pieri's rule (\cite{stanley_fomin_1999}).
\end{proof}

For a skew partition $\lambda/\mu$, let $\lambda/\mu -e_j$ (respectively $\lambda/\mu - e'_j$) represent $\lambda/\mu$ with the rightmost box in the $j$th row (respectively the lowest box in the $j$th column) removed if such a box exists, or $\lambda/\mu$ otherwise.  Note that this may no longer be a skew partition.  The following is a well-known fact (recall our convention that $f^\alpha=0$ unless $\alpha$ is a valid skew partition):

\begin{proposition}
\label{prop:SYTReduction}
For any skew partition $\lambda/\mu$, \[f^{\lambda/\mu} = \sum_{j} f^{\lambda/\mu-e_j} = \sum_{j} f^{\lambda/\mu-e'_j}\]
\end{proposition}
Finally, we recall the definition of the \emph{Frobenius characteristic map} (see e.g. \cite{macdonald}). This definition will be used in Section 6 when we give a representation-theoretic interpretation of the column case. We recall some basic facts about symmetric functions. The \emph{power sum symmetric function} is defined by
\[
p_r(x_1, x_2, \dots) := x_1^r + x_2^r + \dots
\]
as a formal power series. For a partition $\rho = (\rho_1, \rho_2, \dots, \rho_l) \vdash n$, define $p_\rho := p_{\rho_1}p_{\rho_2} \dots p_{\rho_l}$ to be the power sum symmetric function corresponding to $\rho$

Next, we have the \emph{Schur function}
\[
s_\rho(x_1, x_2, \dots) := \sum_{T \in {\rm SSYT}(\rho)} x^{wt(T)}
\]
Let $z_\rho = \frac{n!}{k_\rho}$, where $k_\rho$ is the number of permutations in $\mathfrak{S}_n$ of cycle type $\rho$.
\begin{definition}
\label{def:frob}
    Let $f$ be the character of a finite-dimensional complex representation $V$ of $\mathfrak{S}_n$. Then the Frobenius characteristic of $f$, written $\text{ch}(f)$ is defined as
    \[
    \text{ch}(f) := \sum_{\rho \vdash n} z^{-1}_{\rho} f(\rho) p_\rho
    \]
    where $f(\rho)$ denotes the value of $f$ on any element of the conjugacy class corresponding to $\rho$. We also use the notation $\text{ch}(V)$ to denote the same symmetric function.
\end{definition}
It is important to note that if $V^\lambda$ is the irreducible representation of $\mathfrak{S}_n$ corresponding to the partition $\lambda$, then 
\[
\text{ch}(V^\lambda) = s_\lambda(x_1, x_2, \dots)
\]
is the corresponding Schur function.

\section{The Column and Row Cases}

\subsection{An alternating formula}

We begin with a lemma which allows us to express the Schur coefficients of $\sbar_\lambda$ as an alternating sum. Let ${\rm Par}_n$ denote the set of partitions with at most $n$ parts.

\begin{lemma}
\label{lemma:alt}
    For $n\in \mathbb{N}$ and $\lambda\in {\rm Par}_n$,
    \begin{align*}
        \sbar_{\lambda}(x_1,\dots, x_n) = \sum_{\nu\subseteq \lambda} (-1)^{|\nu|} s_{\nu}(x_1,\dots, x_n)s_{(1)}^{|\lambda/\nu|}(x_1,\dots, x_n)f^{\lambda/\nu}\frac{\prod_{(r,c)\in \lambda/\nu}(n-r+c)}{|\lambda/\nu|!}
    \end{align*}
    where $(r, c) \in \lambda/\nu$ refers to the ordered pair denoting the row and column index of a box in the skew shape $\lambda/\nu$.
\end{lemma}

\begin{proof}
    First consider $s_\lambda(t+x_1, t+x_2, \dots, t+x_n)$, where $t$ is a new indeterminate. Taking the Taylor series of this function with respect to $t$ about $t=0$, and grouping terms by power of $t$, we get
    \begin{align*}
        s_\lambda(t+x_1, \dots, t+x_n) = \sum_{k=0}^n t^k  \frac{\nabla^k}{k!} (s_\lambda(x_1, \dots, x_n))
    \end{align*}
where we define $\nabla f(x_1, \dots, x_n)\ := \sum_{i=1}^n \frac{\partial f}{\partial x_i}$, and $\nabla^k$ indicates the operator $\nabla$ applied $k$ times. We continue with 

    \begin{Claim}
    \label{deriv}
    \[
    \frac{\nabla^k}{k!}(s_\lambda(x_1, \dots, x_n)) = \sum_{ \genfrac{}{}{0pt}{}{\nu \subseteq \lambda}{|\nu| = |\lambda| - k}} s_\nu(x_1, \dots, x_n) f^{\lambda / \nu} \frac{\prod_{(r, c) \in \lambda / \nu} (n-r+c)}{|\lambda / \nu|!}
    \]
    \end{Claim}

    \begin{proof}
        We first consider the case $k=1$. We use the formula in Equation $\ref{eq:bialt}$ for the Schur polynomial $s_\lambda(x_1, \dots, x_n)$.

        Now note that since $\nabla$ is a symmetric differential operator on the polynomial ring, it commutes with the action of $\varepsilon$. Then by the quotient rule
        \begin{align*}
            \nabla s_\lambda(x_1, \dots, x_n) & = \nabla \left(\frac{\varepsilon \cdot x^{\lambda + \delta}}{\varepsilon \cdot x^\delta} \right) \\
            & = \frac{(\varepsilon \cdot x^\delta)(\nabla(\varepsilon \cdot x^{\lambda + \delta}))-(\varepsilon \cdot x^{\lambda+\delta})(\nabla(\varepsilon \cdot x^\delta))}{(\varepsilon \cdot x^\delta)^2}\\
            & = \frac{(\varepsilon \cdot x^\delta)(\varepsilon \cdot \nabla( x^{\lambda + \delta}))-(\varepsilon \cdot x^{\lambda+\delta})(\varepsilon \cdot \nabla x^\delta)}{(\varepsilon \cdot x^\delta)^2}\\
            & = \frac{\varepsilon \cdot \nabla (x^{\lambda + \delta}) }{\varepsilon \cdot x^\delta}\\
            & = \frac{\varepsilon \cdot \left( \sum_{i=1}^n (\lambda_i + \delta_i) x_1^{\lambda_1+\delta_1}\dots x_i^{\lambda_i + \delta_i - 1}\dots x_n^{\lambda_n + \delta_n}\right)}{\varepsilon \cdot x^\delta}\\
            & = \frac{\varepsilon \cdot \left( \sum_{i=1}^n (\lambda_i + n - i) x_1^{\lambda_1+\delta_1}\dots x_i^{\lambda_i + \delta_i - 1}\dots x_n^{\lambda_n + \delta_n}\right) }{\varepsilon \cdot x^\delta}
        \end{align*}
        In the fourth equality above, $\varepsilon$ annihilates $\nabla x^\delta$. In general $\varepsilon$ will annihilate any term in the numerator which has a repeated exponent. The remaining terms in the last expression are the monomials that correspond to removing an outside corner from the Young diagram of $\lambda$. We conclude that 
        \begin{align*}
            \nabla s_\lambda(x_1, \dots, x_n) = \sum_{\nu=\lambda \setminus \text{(r, c)}} (n-r+c) s_\nu(x_1,\dots, x_n)
        \end{align*}
        where we sum over all $(r, c)$ which are outside corners of $\lambda$. By iterating this argument, we can see that 
        \[
        \nabla^k(s_\lambda(x_1, \dots, x_n)) = \sum_{\genfrac{}{}{0pt}{}{\nu \subseteq \lambda}{|\nu| = |\lambda| - k}} s_\nu(x_1, \dots, x_n) f^{\lambda / \nu} \prod_{(r, c) \in \lambda / \nu} (n-r+c)
        \]
        where the factor of $f^{\lambda / \nu}$ is the number of ways of removing $k$ outside corners from $\lambda$ in succession. Dividing both sides by $k!$ completes the proof of the claim.
    \end{proof}
    The proof of the lemma follows by substituting $t = -s_{(1)}(x_1, \dots, x_n)$ and multiplying through by $(-1)^{|\lambda|}$.
\end{proof}

We have the following corollary, which will prove useful in the next two subsections.

\begin{corollary}
\label{cor:alt}
    For $n\in \mathbb{N}$ and $\lambda,\mu\in {\rm Par}_n$,
    \begin{align*}
        [s_{\mu}]\sbar_{\lambda}(x_1,\dots, x_n) &= \sum_{\nu\subseteq \lambda\cap \mu} (-1)^{|\nu|} \frac{f^{\lambda/\nu}f^{\mu/\nu}\prod_{(r,c)\in \lambda/\nu}(n-r+c) }{|\lambda/\nu|!}\\
        &= \sum_{\nu\subseteq \lambda\cap \mu} (-1)^{|\nu|} \binom{|\lambda|}{|\nu|}f^{\mu/\nu}\frac{f^\nu f^{\lambda/\nu}}{f^\lambda}\frac{|{\rm SSYT}(\lambda,n)|}{|{\rm SSYT}(\nu,n)|}
    \end{align*}
\end{corollary}

\begin{proof}
The first equality follows from Proposition \ref{prop:boxmult} applied to each term in Lemma \ref{lemma:alt}. The second equality follows from applying both the hook length formula for standard Young tableaux and the hook content formula for semistandard Young tableaux (see \cite{stanley_fomin_1999}).

\end{proof}

\subsection{The column case}

Denote by $P(\lambda,n)$ the set of elements of ${\rm SYT}(\lambda,n)$ such that the smallest label not appearing in the first column is odd. We call the elements of $P(\lambda, n)$ \emph{parity tableaux}. The following theorem gives a combinatorial interpretation for the Schur coefficients of $\sbar_{(1^k)}(x_1,\dots, x_n)$, generalizing Proposition 2.3 in \cite{reinerwebb} to the cases $k \leq n$.

\begin{theorem}
\label{thm:columncase}
\begin{align*}
    \sbar_{(1^k)}(x_1,\dots, x_n) = \sum_{|\mu| = k} |P(\mu,n)| s_\mu(x_1,\dots, x_n)
\end{align*}
\end{theorem}
\begin{proof}
We first apply Corollary \ref{cor:alt}. Note that if $\nu \subseteq \lambda \cap \mu$, then $\nu$ is a column, say of length $i$. Then the terms $f^{\lambda / \nu}$, $f^{\nu}$, and $f^{\lambda}$ are all equal to $1$.

Next, we have $|{\rm SSYT}(\lambda, n)| = \binom{n}{k}$, $|{\rm SSYT}(\nu, n)| = \binom{n}{i}$, and $\binom{|\lambda|}{|\nu|} = \binom{k}{i}$. So we have $\binom{k}{i} \cdot \binom{n}{k} / \binom{n}{i} = \binom{n-i}{n-k}$. We conclude that
\begin{align*}
    [s_\mu]\sbar_{(1^k)}(x_1, \dots, x_n) = \sum_{i=0}^k (- 1)^i \binom{n-i}{k-i} f^{\mu/ (1^i)}
\end{align*}

Note that $\binom{n-i}{k-i} f^{\mu/ (1^i)}$ counts the elements of ${\rm SYT}(\mu/ (1^i),n-i)$ - or equivalently, elements of ${\rm SYT}(\mu,n)$ such that the numbers from $1$ through $i$ are all in the leftmost column. 

This implies that 
\[
\binom{n-2j}{k-2j} f^{\mu/ (1^{2j})}- \binom{n-2j-1}{k-2j-1} f^{\mu/ (1^{2j+1})}
\]
counts the number of elements of ${\rm SYT}(\mu,n)$ such that the numbers from $1$ through $2j$ are all in the first column, but $2j+1$ is not, so $2j+1$ is the smallest number not appearing in the first column. 

As a result, summing over all $i$, we get that $\sum_{i=0}^k (-1)^i \binom{n-i}{k-i} f^{\mu/(1^i)}$ is the number of elements of ${\rm SYT}(\mu,n)$ such that the smallest number number not appearing in the first column is odd. This recovers the definition of $P(\mu,n)$, so we are done.
\end{proof}

\subsection{The row case}
We now turn to examining the case $\sbar_{(k)}(x_1, \dots, x_n)$. This case will prove to be somewhat more complex. Given a partition $\mu$ and an integer $p$ such that $p+|\mu| \geq 0$, we define a new partition $\mu^{(p)}$ as follows: Let 
\begin{equation*}
    a = \min\{i\in \mathbb{N}: p\geq \mu'_i-i\}
\end{equation*}
Then the $i$th column of $\mu^{(p)}$ is given by:
\[(\mu^{(p)})_i' = \begin{cases} 
      \mu_i'-1 & i<a \\
      p+a-1 & i=a \\
      \mu_{i-1}' & i>a 
   \end{cases}
\]
We proceed with the following lemma:
\begin{lemma}
\label{thm:rowcasegeneral}
Fix a partition $\mu$ and $p\in \mathbb{Z}$ such that $p+|\mu|\geq 0$.
Then \[ (-1)^{a+1}f^{\mu^{(p)}} = \sum_{i=0}^{|\mu|} (-1)^i \binom{p + |\mu|}{p + i}f^{\mu/ (i)}\]
In particular, if $p = \mu'_a - a$, then 
\[ 0 = \sum_{i=0}^{|\mu|} (-1)^i \binom{p + |\mu|}{p + i}f^{\mu/ (i)}.\]
\end{lemma}
\begin{proof}
By the definition of $a$, we have that the sequence $(\mu'_i - i)_{i\in \mathbb{N}}$ is strictly decreasing, and $\mu^{(p)}$ is a partition if and only if $p > \mu'_a - a$. We proceed by induction on $|\mu|+p$.  In the base case, $|\mu|+p=0$, which implies that
\[
\sum_{i=0}^{|\mu|} (-1)^i \binom{p + |\mu|}{p + i}f^{\mu/ (i)} = (-1)^{|\mu|}f^{\mu/ (|\mu|)}
\]
If $\mu\not= (|\mu|)$, then $\mu'_{|\mu|} = 0$ so we have that $p  = -|\mu| = -|\mu| +0 = -|\mu| + \mu'_{|\mu|}$ which implies that $a = |\mu|$, and so $(\mu^{(-|\mu|)})'_{|\mu|} = p+a-1 = -1$, so $\mu^{(-|\mu|)}$ is not a partition, so $f^{\mu^{(-|\mu|)}} = 0$.  Similarly, if $\mu \not= (|\mu|)$, then $f^{\mu/ (|\mu|)} = 0$, so $\sum_{i=0}^{|\mu|} (-1)^i \binom{p + |\mu|}{p + i}f^{\mu/ (i)} =0 = f^{\mu^{(-|\mu|)}}$, completing the base case in this case.

On the other hand, if $\mu = (|\mu|)$, then $\mu'_{|\mu|} - |\mu| = 1-|\mu| > p > -|\mu| - 1 = \mu'_{|\mu|+1} - (|\mu|+1)$, so $a = |\mu|+1$.  Therefore,  $(|\mu|)^{(-|\mu|)} = (0)$, and indeed
\begin{align*}
    \sum_{i=0}^{|\mu|} (-1)^i \binom{p + |\mu|}{p + i}f^{(|\mu|)/ (i)} & = (-1)^{|\mu|}f^{(|\mu|)/ (|\mu|)} \\
    & = (-1)^{|\mu|} \\
    & = (-1)^{|\mu|}f^{(0)} = (-1)^{-|\mu|-1+1}f^{(|\mu|)^{(-|\mu|)}}
\end{align*}

which completes the proof of the base case.

For the induction step, assume $p+|\mu|>0$, and assume for our inductive hypothesis that the statement is true for smaller values of $p + |\mu|$.  Using Proposition \ref{prop:SYTReduction} and a well-known identity of binomial coefficients, we have that 
\begin{align*}
    \sum_{i=0}^{|\mu|} (-1)^i \binom{p + |\mu|}{p + i}f^{\mu/ (i)} 
    &= \sum_{i=0}^{|\mu|} (-1)^i \left(\binom{p + |\mu|-1}{p + i-1} + \binom{p + |\mu|-1}{p + i}\right)f^{\mu/ (i)}\\
    &= \sum_{i=0}^{|\mu|} (-1)^i \binom{p + |\mu|-1}{p + i-1} f^{\mu/ (i)} + \sum_{i=0}^{|\mu|} (-1)^i  \binom{p + |\mu|-1}{p + i}f^{\mu/ (i)}\\
    &= \sum_{i=0}^{|\mu|} (-1)^i \binom{(p-1) + |\mu|}{( p -1) + i} f^{\mu/ (i)} \\&
    \hspace{1cm}+\sum_{i=0}^{|\mu|-1} (-1)^i  \binom{p + (|\mu|-1)}{p + i}f^{\mu/ (i)}\\
    &= \sum_{i=0}^{|\mu|} (-1)^i \binom{(p-1) + |\mu|}{( p -1) + i} f^{\mu/ (i)} \\& \hspace{1cm}+\sum_{i=0}^{|\mu|-1} (-1)^i \binom{p + (|\mu|-1)}{p + i} \left(\sum_{j=1}^\infty f^{(\mu/ (i))-e'_j}\right)\\
    &= \sum_{i=0}^{|\mu|} (-1)^i \binom{(p-1) + |\mu|}{( p -1) + i} f^{\mu/ (i)} \\& \hspace{1cm}+ \sum_{j=1}^\infty\sum_{i=0}^{|\mu|-1} (-1)^i  \binom{p + (|\mu|-1)}{p + i}f^{(\mu/ (i))-e'_j}\\
    &= (-1)^{a(\mu,p-1)+1} f^{\mu^{(p-1)}} + \sum_{j=1}^\infty (-1)^{a(\mu-e_j,p)+1} f^{(\mu-e'_j)^{(p)}}
\end{align*}
where we define $a(\mu,p) = \min\{i\in \mathbb{N}: p\geq \mu'_i-i\}$, and in the last step we have used the inductive hypothesis.  We break this up into two cases, depending on whether or not $p= \mu'_{a(\mu,p)}-a(\mu,p)$.

If $p= \mu'_{a(\mu,p)}-a(\mu,p)$, then by construction $a(\mu,p-1) = a(\mu,p) +1$, and for all $j$ such that $\mu-e'_j$ is a partition, $a(\mu-e'_j,p) = a(\mu,p) $.  Additionally, observe that for all $j \not= a(\mu,p)$, $f^{\mu^{(p)}} = f^{(\mu-e'_j)^{(p)}} = 0$, and $\mu^{(p-1)} = (\mu-e'_{a(\mu,p)})^{(p)}$. As a result, we have that 
\begin{align*}
    \sum_{i=0}^{|\mu|} (-1)^i \binom{p + |\mu|}{p + i}f^{\mu/ (i)} &= (-1)^{a(\mu,p-1)+1} f^{\mu^{(p-1)}} + \sum_{j=1}^\infty (-1)^{a(\mu-e'_j,p)+1} f^{(\mu-e'_j)^{(p)}}\\
    &= (-1)^{a(\mu,p)} (f^{\mu^{(p-1)}} -  f^{(\mu-e'_{a(\mu,p)})^{(p)}}) \\
    & = (-1)^{a(\mu,p)} (f^{\mu^{(p-1)}} -  f^{\mu^{(p-1)}})\\
    &= 0 \\
    & = (-1)^{a(\mu,p)+1}f^{\mu^{(p)}},
\end{align*}
completing the induction in this case.

Finally, we consider the case where $p\not= \mu'_{a(\mu,p)}-a(\mu,p)$. This means that \\$p> \mu'_{a(\mu,p)}-a(\mu,p)$, and in particular $a(\mu,p-1) = a(\mu,p)$, with $\mu^{(p-1)} = \mu^{(p)} - e'_{a(\mu,p)}$.  For all $j$, either $a(\mu-e'_j,p) = a(\mu,p)$, or $a(\mu-e'_j,p)=a(\mu,p)-1$.  In the former case, $(\mu-e'_j)^{(p)} = \mu^{(p)} - e'_j$ for $j< a(\mu,p)$ and $(\mu-e'_j)^{(p)} = \mu^{(p)} - e'_{j+1}$ for $j\geq a(\mu,p)$.  On the other hand, if $a(\mu-e'_j,p)=a(\mu,p)-1$, then $j=a(\mu,p)-1$, with $p = \mu'_j - j - 1$, so in this case $(\mu-e'_j)^{(p)}$ is not a partition, so $f^{(\mu-e'_j)^{(p)}} = 0 = f^{\mu^{(p)}-e'_j}$.  Combining these two cases, we see that $\sum_j f^{(\mu-e'_j)^{(p)}} = \sum_{j\not= a(\mu,p)} f^{\mu^{(p)}-e'_j}$.  As a result,
\begin{align*}
    \sum_{i=0}^{|\mu|} (-1)^i \binom{p + |\mu|}{p + i}f^{\mu/ (i)} &= (-1)^{a(\mu,p-1)+1} f^{\mu^{(p-1)}} + \sum_{j=1}^\infty (-1)^{a(\mu-e_j,p)+1} f^{(\mu-e'_j)^{(p)}}\\
    &= (-1)^{a(\mu,p)+1} f^{\mu^{(p-1)}} + \sum_{j=1}^\infty (-1)^{a(\mu,p)+1} f^{(\mu-e'_j)^{(p)}}\\
    &= (-1)^{a(\mu,p)+1} f^{\mu^{(p)}-e'_{a(\mu,p)}} + \sum_{j\not= a(\mu,p)}(-1)^{a(\mu,p)+1} f^{\mu^{(p)}-e'_j}\\
    &= \sum_{j}(-1)^{a(\mu,p)+1} f^{\mu^{(p)}-e'_j}\\
    &= (-1)^{a(\mu,p)+1} f^{\mu^{(p)}}
\end{align*}
completing the proof by induction in this case.
\end{proof}
We now have the tools we need to provide a combinatorial interpretation for the Schur expansion in the row case.
\begin{theorem}
\label{thm:rowcase}
\begin{align*}
    \sbar_{(k)}(x_1,\dots, x_n) = \sum_{\substack{|\mu| = k\\ \ell(\mu)< n}} f^{\mu + (1^{n-1})} s_{\mu}(x_1,\dots, x_n)
\end{align*}
\end{theorem}
\begin{proof}
We again use Corollary \ref{cor:alt}. Note that if $\nu \subseteq \lambda \cap \mu$, then $\nu$ is a row, say of length $i$.

Then we have that $f^{\lambda / \nu}$, $f^{\nu}$, and $f^{\lambda}$ are all equal to $1$, and $\binom{|\lambda|}{|\nu|} = \binom{k}{i}$, $|{\rm SSYT} (\lambda, n)| = \binom{n+k-1}{k}$, and $|{\rm SSYT} (\nu, n)| = \binom{n+i-1}{i}$. We conclude that
\begin{equation}
\label{eq:rowprelim}
    [s_\mu]\sbar_{(k)}(x_1, \dots, x_n) = \sum_{i=0}^k (- 1)^i \binom{n+k-1}{n+i-1} f^{\mu/ (i)}
\end{equation}
By Theorem \ref{thm:rowcasegeneral}, letting $p=n-1$, when $\ell(\mu)= n$ the right-hand side of \ref{eq:rowprelim} is zero, and when $\ell(\mu)< n$ the right-hand side of \ref{eq:rowprelim} is equal to $f^{\mu+(1^{n-1})}$, and so we are done.
\end{proof}

The following definition will be used in Corollary \ref{cor:tworows}. There, we will consider general $\lambda$. For now, we note that this definition gives us another combinatorial interpretation of the row case.
\begin{definition}
\label{def:G}
    For $\lambda,\mu,n$, with $|\lambda| = |\mu|$, define $G(\lambda,\mu,n)$ to be the set of pairs of tableaux \\$(S,T)\in {\rm SSYT}(\lambda,n)\times {\rm SYT}(\mu)$ such that if $i$ appears in row $r$ of $T$, then the $i^{th}$ smallest value in $S$ is greater than $r$. Here, ties for $i^{th}$ smallest value are broken by first declaring identical entries in a lower row to be larger, and then declaring identical entries in a further-right column to be larger. This definition extends to the case where $\lambda$ and/or $\mu$ are replaced with skew partitions.
\end{definition}
\begin{example}
\label{ex:G}
    Let $\lambda = (4)$, $\mu = (2,2)$, and $n=3$. The $G(\lambda, \mu, n)$ consists of the following five pairs of tableaux:
    \begin{gather*}
        \left( \hspace{.1cm}
        \begin{ytableau}
        2 & 2 & 3 & 3
        \end{ytableau} \hspace{.25cm} ,  \hspace{.25cm}
        \begin{ytableau}
            1 & 2\\
            3 & 4
        \end{ytableau} \hspace{.1cm}
        \right)  \hspace{1cm}
        \left( \hspace{.1cm}
        \begin{ytableau}
        2 & 3 & 3 & 3
        \end{ytableau} \hspace{.25cm} , \hspace{.25cm}
        \begin{ytableau}
            1 & 2\\
            3 & 4
        \end{ytableau} \hspace{.1cm}
        \right)\\ \\
        \left( \hspace{.1cm}
        \begin{ytableau}
        3 & 3 & 3 & 3
        \end{ytableau} \hspace{.25cm} ,  \hspace{.25cm} \begin{ytableau}
            1 & 2\\
            3 & 4
        \end{ytableau} \hspace{.1cm}
        \right)  \hspace{1cm}
        \left( \hspace{.1cm}
        \begin{ytableau}
        2 & 3 & 3 & 3
        \end{ytableau} \hspace{.25cm} , \hspace{.25cm} \begin{ytableau}
            1 & 3\\
            2 & 4
        \end{ytableau} \hspace{.1cm}
        \right)\\ \\
          \left( \hspace{.1cm}
        \begin{ytableau}
        3 & 3 & 3 & 3
        \end{ytableau} 
        \hspace{.25cm} , \hspace{.25cm} \begin{ytableau}
            1 & 3\\
            2 & 4
        \end{ytableau} \hspace{.1cm}
        \right)
    \end{gather*}
            
\end{example}
An alternative interpretation of the row case is given by the following corollary.
\begin{corollary}
\label{cor:rowcasealt}
For any positive integers $k$ and $n$, and any partition $\mu$ with $|\mu| = k$,
\begin{align*}
    [s_\mu]\sbar_{(k)}(x_1,\dots, x_n) = |G((k),\mu,n)|
\end{align*}
\end{corollary}
\begin{proof}
We construct a bijection from ${\rm SYT}(\mu + (1^{n-1}))$ to $G((k),\mu,n)$ to prove this statement.  Fix an element $S\in {\rm SYT}(\mu + (1^{n-1}))$, and let $1= a_1<\dots < a_{n-1}$ be the values in the leftmost column of $S$, and let $a_n = n+|\mu| = n+k$.  Let $S'$ be the unique element of ${\rm SSYT}((k),n)$ such that, for each $2\leq i \leq n$, $S'$ has exactly $a_i-a_{i-1}-1$ $i$'s.  Define $T$ to be the unique element of ${\rm SYT}(\mu)$ such that the labels of $T$ are in the same relative order as $S$ with the first column deleted.  This transformation from $S$ to $(S',T)$ is clearly an injection, and the inverse transformation is also an injection, so this is a bijection. 

For example, the pairs in Example $\ref{ex:G}$ correspond respectively to the standard Young tableaux

\begin{gather*}
        \begin{ytableau}
        1 & 2 & 3\\
        4 & 5 & 6
        \end{ytableau}
          \hspace{1cm}
        \begin{ytableau}
        1 & 2 & 4\\
        3 & 5 & 6
        \end{ytableau}\\ \\
        \begin{ytableau}
        1 & 3 & 4\\
        2 & 5 & 6
        \end{ytableau}
          \hspace{1cm}
        \begin{ytableau}
        1 & 2 & 5\\
        3 & 4 & 6
        \end{ytableau}\\ \\
        \begin{ytableau}
        1 & 3 & 5\\
        2 & 4 & 6
        \end{ytableau}
    \end{gather*}

of shape $(2, 2) + (1,1) = (3,3)$.
\end{proof}

The following corollary will be used in the proof of Theorem \ref{thm:singlecolumn}.

\begin{corollary}
\label{cor:rowcolumnmult}
    For any $k\leq n$ and $1\leq p \leq n-k+1$,
    \begin{align*}
    [s_{(1^k)}]&\sbar_{(k)}(x_1,\dots, x_n) \\
    &= |\{S\in {\rm SSYT}((k),n): S(1,1)>1, (S(1,i),S(1,i+1))\not= (p+i,p+i)\}|
\end{align*}
\end{corollary}
\begin{proof}
We construct a bijection from the desired set to $G((k),(1^k),n)$.  Let \\$S\in {\rm SSYT}((k),n)$ be such that $S(1,1) > 1$ and $(S(1,i),S(1,i+1))\not= (p+i,p+i)$ for any $1\leq i < k$.  Break the bijection into two cases depending on whether there exists some $i$ for which $S(1,i) \leq p+i-1$.

If $S(1,i) > p+i-1$ for all $i$, then the bijection maps $S$ to $(S,T)$, where $T$ is the unique element of ${\rm SYT}((1^k))$. To show that it is well-defined, observe that since $S(1,i) > p+i-1$ for all $i$, in particular $S(1,i) > i$, so $(S,T) \in G((k),(1^k),n)$.

On the other hand, if there exists some $i$ for which $S(1,i) \leq p+i-1$, let $a = \max\{i: S(1,i) \leq p+i-1\}$.  Let $S'\in {\rm SSYT}((k),n)$ be defined as $S'(1,i) = S(1,i)$ for $i>a$, and $S'(1,i) = p+a+1 - S(1,a+1-i)$ for $i\leq a$.  The bijection will map $S$ to $(S',T)$, where $T$ is the unique element of ${\rm SYT}((1^k))$.

For example, suppose $k=3$ and $n=4$. First let $p=1$, so the elements of the set in question are
\begin{gather*}
    \begin{ytableau}
        2 & 3 & 4
    \end{ytableau} \hspace{1cm}
    \begin{ytableau}
        2 & 4 & 4
    \end{ytableau} \hspace{1cm}
    \begin{ytableau}
        3 & 3 & 4
    \end{ytableau} \\
    \begin{ytableau}
        3 & 4 & 4
    \end{ytableau} \hspace{1cm}
    \begin{ytableau}
        4 & 4 & 4
    \end{ytableau}
\end{gather*}
For each of these tableau $S$, we have $S(1,i) > p+i-1 = i$, so the bijection will map $S$ to $(S, T)$, where $T$ is the unique element of ${\rm SYT}((1^k))$.
Now suppose instead that $p=2$. Then the mapping of $S$ to $S'$ looks as follows:
\begin{center}
\begin{tabular}{c | c}
$S$ & $S'$ \\
\hline
\rule{0pt}{4ex}    
\begin{ytableau}
        2 & 2 & 2
    \end{ytableau} & \begin{ytableau}
        4 & 4 & 4
    \end{ytableau} \\
    \hline
\rule{0pt}{4ex} 
    \begin{ytableau}
        2 & 2 & 3
    \end{ytableau} & \begin{ytableau}
        3 & 4 & 4
    \end{ytableau} \\
    \hline

\rule{0pt}{4ex} 
    \begin{ytableau}
        2 & 2 & 4
    \end{ytableau} & \begin{ytableau}
        2 & 4 & 4
    \end{ytableau} \\
    \hline

    \rule{0pt}{4ex} 
    \begin{ytableau}
        2 & 3 & 3
    \end{ytableau} & \begin{ytableau}
        3 & 3 & 4
    \end{ytableau} \\
    \hline

\rule{0pt}{4ex} 
    \begin{ytableau}
        2 & 3 & 4
    \end{ytableau} & \begin{ytableau}
        2 & 3 & 4
    \end{ytableau}
    
\end{tabular}
\end{center}
We can see that the elements $(S', T)$ are again the elements of $G((3),(1^3), 4)$, and so the bijection will map the tableau $S$ to the pair $(S', T)$. We proceed by proving that this map is indeed a bijection. We start with a claim.

\begin{claim}
\label{claim:rowbound}
For all $i \leq a$, $S(1,i)\leq p+i-1$.
\end{claim}

\begin{proof}
Assume for the sake of contradiction that there exists $i \leq a$ such that $S(1,i)> p+i-1$, and let $j$ be the maximal such $i$.  By the definition of $a$, $j<a$, and so 
\[
p+j \geq S(1,j+1) \geq S(1,j) > p+j-1
\]
so $S(1,j) = S(1,j+1) = p+j$, which contradicts the construction of $S$.
\end{proof}

Next, we show that the map is well-defined in this case.  Since $S$ is weakly increasing along the row, the only cells where it is not obvious that $S'$ is weakly increasing is from $S'(1,a)$ to $S'(1,a+1)$.  For here, observe that since $S(1,1)\geq 2$ and $S'(1,a+1) > p+a$, that $S'(1,a) \leq p+a+1-2 = p+a-1 < p+a+1< S'(1,a+1)$. By Claim~\ref{claim:rowbound}, for all $i\leq a$ we have that $S'(1,i) = p+a+1 -S(1,a+1-i) \geq p+a+1 - (p + (a+1-i)-1) = i+1 >i$.  Also, for $i > a$, $S'(1,i) \geq p+i > i$.  As a result, this map is well-defined.

To show that this map is a bijection, it suffices to show that it is invertible.  The inverse map is defined as follows: given $(S',T)\in G((k),(1^k),n) $, return $S$ if $S'(1,i) > p+i -1$ for all $i$.  If there exists some $i$ such that $ S'(1,i) \leq p+i-1$, then let $a = \max\{i: S'(1,i) \leq p+i-1\}$.  In this case, the map returns $S\in {\rm SSYT}((k))$ such that $S(1,i) = S'(1,i)$ for all $i>a$, and $S(1,i) = p+a+1 - S'(1,a+1-i)$.  In the first case, since $S'(1,i) > p+i -1$, in particular $S'(1,i+1) \not=p+i$ for any $i$, so $(S(1,i),S(1,i+1)) \not= (p+i,p+i)$, so this map is well defined.

For the second case, since $S'$ is weakly increasing, the only place where it is not obvious that $S$ is weakly increasing is from $S(1,a)$ to $S(1,a+1)$.  However, here 
\[S(1,a+1) > p + (a+1) -1 = p+ a = p +a + 2 - 2 > p+a +1 - S'(1,a) = S(1,a)
\]
so it is weakly increasing everywhere.  Also, notice that for all $i>a$, $S'(1,i) > p+i -1$, so in particular $S'(1,i) \not=p+i-1$ for any $i$, so $(S(1,i-1),S(1,i)) \not= (p+i-1,p+i-1)$.  Similarly, for all $i\leq a$, $S'(1,i)>i$, so 
\[
S(1,i) = p+a+1 - S'(1,a+1-i) < p+a+1 - (a+1-i) = p+i
\]
so in particular $S'(1,i) \not= p+i$. so $(S(1,i),S(1,i+1)) \not= (p+i,p+i)$.  This implies that the reverse map is well-defined.

The fact that these two maps are inverses of each other follows from the observation that the two $a$-values in the maps will be the same, so the map is a bijection. So the two sets in question are the same size, completing the proof.

\end{proof}

\section{Multiplicity of one or two rows}

This section and the next provide some partial results on the Schur expansion of $\sbar_{\lambda}$ for general $\lambda$. In particular, we compute the coefficients of $s_{(a, b)}$ and $s_{1^{|\lambda|}}$.

For a $T\in {\rm SSYT}(\lambda)$, let $wt_{>i}(T) = \sum_{j>i} wt_i(T)$. We start with the following lemma.

\begin{lemma}
\label{lemma:2RowMonomial}
For $a\geq b\geq 0$ such that $a+b = |\lambda|$,
\begin{align*}
    [x_1^a x_2^b] \sbar_{\lambda}(x_1,\dots,x_n) = \sum_{\substack{\nu\in {\rm Par}_2 \\ |\nu|\leq a+b}} |{\rm SSYT}(\lambda/\nu,[3,n])|\sum_{d = \nu_2}^{\nu_1}\binom{a+b-|\nu|}{a-d}.
\end{align*}
\end{lemma}
\begin{proof} By the definition of $\sbar_\lambda$ and the binomial theorem, we have
    \begin{align*}
        [x_1^a x_2^b] \sbar_{\lambda}(x_1,\dots,x_n) &= [x_1^a x_2^b]\sum_{T\in {\rm SSYT}(\lambda)} \prod_{i=1}^n(x_1 + \dots + x_{i-1} + x_{i+1} + \dots + x_n)^{wt_i(T)}\\
        &= \sum_{T\in {\rm SSYT}(\lambda)} [x_1^a x_2^b]x_2^{wt_1(T)} x_1^{wt_2(T)} (x_1+x_2)^{wt_{>2}(T)}\\
        &= \sum_{T\in {\rm SSYT}(\lambda)} [x_1^{a-wt_2(T)} x_2^{b-wt_1(T)}](x_1+x_2)^{wt_{>2}(T)}\\
        &= \sum_{T\in {\rm SSYT}(\lambda)} \binom{wt_{>2}(T)}{a-wt_2(T)}\\
        &= \sum_{c+d = 0}^{a+b}\sum_{\substack{T\in {\rm SSYT}(\lambda) \\ wt_1(T) = c, wt_2(T) = d}}
        \binom{wt_{>2}(T)}{a-wt_2(T)}\\
        &= \sum_{c+d = 0}^{a+b} \binom{a+b-c-d}{a-d}|\{T\in {\rm SSYT}(\lambda): wt_1(T) = c, wt_2(T)=d\}|\\
    \end{align*}
For $c,d$, an element $T\in {\rm SSYT}(\lambda)$ with  $wt_1(T) = c, wt_2(T)=d$ must have all $1$'s in the first row, and can have at most $\min(c,d)$ $2$'s in the second row, with the rest in the first row to the right of the $1$'s.  Beyond that, the numbers larger than 2 can be arranged in any way, so
\[
|\{T\in {\rm SSYT}(\lambda): wt_1(T) = c, wt_2(T)=d\}| = \sum_{i = \max(c,d)}^{c+d} |{\rm SSYT}(\lambda/(i,c+d-i),[3,n])|
\]
As a result,
    \begin{align*}
        [x_1^a x_2^b] \sbar_{\lambda}(x_1,\dots,x_n) &= \sum_{c+d = 0}^{a+b} \binom{a+b-c-d}{a-d}|\{T\in {\rm SSYT}(\lambda): wt_1(T) = c, wt_2(T)=d\}|\\
        &= \sum_{c+d = 0}^{a+b} \binom{a+b-c-d}{a-d}\sum_{i = \max(c,d)}^{c+d} |{\rm SSYT}(\lambda/(i,c+d-i),[3,n])|\\
        &= \sum_{k = 0}^{a+b} \sum_{d = 0}^k\binom{a+b-k}{a-d}\sum_{i = \max(k-d,d)}^{k} |{\rm SSYT}(\lambda/(i,k-i),[3,n])|\\
        &= \sum_{k = 0}^{a+b} \sum_{d = 0}^k\sum_{i = \max(k-d,d)}^{k}\binom{a+b-k}{a-d} |{\rm SSYT}(\lambda/(i,k-i),[3,n])|\\
        &= \sum_{k = 0}^{a+b} \sum_{i = \lceil k/2 \rceil}^{k}\sum_{d = k-i}^i\binom{a+b-k}{a-d} |{\rm SSYT}(\lambda/(i,k-i),[3,n])|\\
        &= \sum_{k = 0}^{a+b} \sum_{i = \lceil k/2 \rceil}^{k}|{\rm SSYT}(\lambda/(i,k-i),[3,n])|\sum_{d = k-i}^i\binom{a+b-k}{a-d} \\
        &= \sum_{\substack{\nu\in {\rm Par}_2 \\ |\nu|\leq a+b}} |{\rm SSYT}(\lambda/\nu,[3,n])|\sum_{d = \nu_2}^{\nu_1}\binom{a+b-|\nu|}{a-d}, 
    \end{align*}
completing the proof.
\end{proof}
As a corollary, we can express the coefficient indexed by a row in the expansion of $\sbar_\lambda$ as a simple count of semistandard Young tableaux.
\begin{corollary}
    \label{cor:singlerow}
    For any partition $\lambda$ and any positive integer $n \geq 2$,
    \begin{align*}
        [s_{(|\lambda|)}]\sbar_{\lambda}(x_1,\dots,x_n) = |{\rm SSYT}(\lambda,[2,n])|
    \end{align*}
\end{corollary}

\begin{proof}
    By Lemma~\ref{lemma:2RowMonomial}, 
    \begin{align*}
        [s_{(|\lambda|)}(x_1,\dots,x_n)] \sbar_{\lambda}(x_1,\dots,x_n) &= [x_1^{|\lambda|} x_2^0] \sbar_{\lambda}(x_1,\dots,x_n)\\
        &= \sum_{\substack{\nu\in {\rm Par}_2 \\ |\nu|\leq |\lambda|}} |{\rm SSYT}(\lambda/\nu,[3,n])|\sum_{d = \nu_2}^{\nu_1}\binom{|\lambda|-|\nu|}{|\lambda|-d}\\
        &= \sum_{\substack{\nu\in {\rm Par}_2 \\ |\nu|\leq |\lambda|}} |{\rm SSYT}(\lambda/\nu,[3,n])|\sum_{d = \nu_2}^{\nu_1}\binom{|\lambda|-|\nu|}{d-|\nu|}\\
        &= \sum_{\substack{\nu\in {\rm Par}_1 \\ |\nu|\leq |\lambda|}} |{\rm SSYT}(\lambda/\nu,[3,n])|\\
        &= |{\rm SSYT}(\lambda,[2,n])|,
    \end{align*}
    completing the proof.
\end{proof}
We are finally able to give a combinatorial expression for the coefficeints indexed by two rows in the expansion of $\sbar_\lambda$. Recall the definition of $G(\lambda,\mu,n)$ (Definition \ref{def:G}).
\begin{corollary}
    \label{cor:tworows}
    For $a\geq b \geq 0$,
        \begin{align*}
        [s_{(a,b)}]\sbar_{\lambda}(x_1,\dots,x_n) = \sum_{p} G(\lambda/(p,p),(a,b)/(p,p),n).   
        \end{align*}
\end{corollary}

\begin{proof}
    By Lemma~\ref{lemma:2RowMonomial}, 
    \begin{align*}
        [s_{(a,b)}]\sbar_{\lambda}(x_1,\dots,x_n) =& [x_1^{a} x_2^b] \sbar_{\lambda}(x_1,\dots,x_n) - [x_1^{a+1} x_2^{b-1}] \sbar_{\lambda}(x_1,\dots,x_n)\\
        =& \sum_{\substack{\nu\in {\rm Par}_2 \\ |\nu|\leq a+b}} |{\rm SSYT}(\lambda/\nu,[3,n])|\sum_{d = \nu_2}^{\nu_1}\binom{a+b-|\nu|}{a-d}\\
        &- \sum_{\substack{\nu\in {\rm Par}_2 \\ |\nu|\leq a+b}} |{\rm SSYT}(\lambda/\nu,[3,n])|\sum_{d = \nu_2}^{\nu_1}\binom{a+b-|\nu|}{a+1-d}\\
        =& \sum_{\substack{\nu\in {\rm Par}_2 \\ |\nu|\leq a+b}} |{\rm SSYT}(\lambda/\nu,[3,n])|\sum_{d = \nu_2}^{\nu_1}\left(\binom{a+b-|\nu|}{a-d}-\binom{a+b-|\nu|}{a+1-d}\right)\\
        =& \sum_{\substack{\nu\in {\rm Par}_2 \\ |\nu|\leq a+b}} |{\rm SSYT}(\lambda/\nu,[3,n])|\left(\binom{a+b-|\nu|}{a-\nu_1}-\binom{a+b-|\nu|}{a+1-\nu_2}\right)\\
        =& \sum_{\substack{\nu\in {\rm Par}_2 \\ |\nu|\leq a+b}} |{\rm SSYT}(\lambda/\nu,[3,n])|f^{(a,b)/\nu}\\
        =& \sum_{\nu \subseteq (a,b)} |{\rm SSYT}(\lambda/\nu,[3,n])|f^{(a,b)/\nu}.
    \end{align*}
    One can think of $|{\rm SSYT}(\lambda/\nu,[3,n])|$ as the number of elements of ${\rm SSYT}(\lambda/(\nu_2,\nu_2),[2,n])$ that have $\nu_1-\nu_2$ $2$'s, all in the first row.  Similarly, one can think of $f^{(a,b)/\nu}$ as the number of elements of ${\rm SYT}((a,b)/(\nu_2,\nu_2))$ such that $1,\dots, \nu_1-\nu_2$ all appear in the first row.  As a result, for a fixed $p\leq b$, $\sum_{\nu: \nu_2=p}|{\rm SSYT}(\lambda/\nu,[3,n])|f^{(a,b)/\nu}$ equals the number of pairs $(S,T)\in {\rm SSYT}(\lambda/(p,p),[2,n])\times {\rm SYT}((a,b)/(p,p))$ such that all of the $2$'s in $S$ are in the first row, and if $i$ is in the second row of $T$, then the $i$th smallest element of $S$ is at least $3$.  Summing over all $p$ gives us the desired answer.
    \end{proof}
We note that this computation relied on the relatively simple relationship between the Schur coefficients indexed by two rows and the monomial coefficients indexed by corresponding powers of $x_1$ and $x_2$ (see the first line of the proof of \ref{cor:tworows}). We were not able to use the same method to compute the coefficients corresponding to $\mu$ for $\mu$ having more than two rows.

\section{Multiplicity of a Column}

In this section we compute the coefficient of $s_{(1^{|\lambda|})}$ in the expansion of $\sbar_\lambda$. This proof uses the Lindstr\"om-Gessel-Viennot lemma for counting lattice paths, (see e.g. \cite{gessel}).
\begin{theorem}
    \label{thm:singlecolumn}
    For any $\ell(\lambda)\leq p\leq n-\lambda_1$,
    $[s_{(1^{|\lambda|})}]\sbar_{\lambda}(x_1,\dots, x_n)$ equals the number of fillings $T$ of $\lambda$ with numbers in $[2,n]$ such that for all $(r,c)\in \lambda$:
    \begin{enumerate}
        \item 
        $T(r,c)\leq T(r+1,c)$, and they are not both equal to $p-r+c$
        \item 
        Either $T(r-1,c)<T(r,c)$ or $T(r-1,c)=T(r,c)=p-r+c$.
    \end{enumerate}
\end{theorem}

\begin{proof}
     By the Jacobi-Trudi formula, 
\begin{align*}
    s_{\lambda}(x_1,\dots, x_n) = \det(s_{(\lambda_i - i + j)}(x_1,\dots, x_n))_{1\leq i,j\leq \ell(\lambda)},
\end{align*}
so 
\begin{align*}
    \sbar_{\lambda}(x_1,\dots, x_n) = \det(\sbar_{(\lambda_i - i + j)}(x_1,\dots, x_n))_{1\leq i,j\leq \ell(\lambda)},
\end{align*}
and so in particular
\begin{equation}
    \label{eq:det1}
    [s_{(1^{|\lambda|})}]\sbar_{\lambda}(x_1,\dots, x_n)
    = \det([s_{(1^{\lambda_i-i+j})}]\sbar_{(\lambda_i - i + j)}(x_1,\dots, x_n))_{1\leq i,j\leq \ell(\lambda)}
\end{equation}

Consider the
directed graph with vertex set $V = [n]\times [n]$, with directed edges from $(x,y)$ to $(x,y+1)$ for all $x,y$, from $(x,x)$ to $(x+1,x+1)$ for each $x>1$, and from $(x,y)$ to $(x+1,y)$ for $y\not=1$ and $x\not=y$. Here is a picture of this planar network when $n=4$:
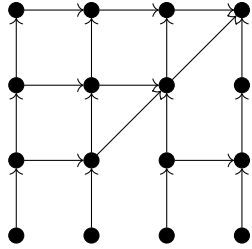
\begin{center}
\begin{tikzpicture}[%
  every node/.style={draw,fill=black,circle,minimum size=.4cm},node distance=1cm]
  \node[scale=.5] (0) at (0,0) {};
  \node[scale=.5] (1) at (0,1) {};
\node[scale=.5] (2) at (0,2) {};
\node[scale=.5] (3) at (0,3) {};
\node[scale=.5] (4) at (1,0) {};
\node[scale=.5] (5) at (1,1) {};
\node[scale=.5] (6) at (1,2) {};
\node[scale=.5] (7) at (1,3) {};
\node[scale=.5] (8) at (2,0) {};
\node[scale=.5] (9) at (2,1) {};
\node[scale=.5] (10) at (2,2) {};
\node[scale=.5] (11) at (2,3) {};
\node[scale=.5] (12) at (3,0) {};
\node[scale=.5] (13) at (3,1) {};
\node[scale=.5] (14) at (3,2) {};
\node[scale=.5] (15) at (3,3) {};

\draw[->] (0) -- (1);
\draw[->] (1) -- (2);
\draw[->] (2) -- (3);
\draw[->] (4) -- (5);
\draw[->] (5) -- (6);
\draw[->] (6) -- (7);
\draw[->] (8) -- (9);

\draw[->] (9) -- (10);
\draw[->] (10) -- (11);
\draw[->] (12) -- (13);
\draw[->] (13) -- (14);
\draw[->] (14) -- (15);
\draw[->] (1) -- (5);
\draw[->] (2) -- (6);

\draw[->] (3) -- (7);
\draw[->] (5) -- (10);
\draw[->] (6) -- (10);
\draw[->] (7) -- (11);
\draw[->] (9) -- (13);
\draw[->] (10) -- (15);
\draw[->] (11) -- (15);

\end{tikzpicture}
\end{center}
Let $P_{a,b}$ be the number of directed paths from $(a,1)$ to $(b,n)$.
\begin{claim}
\label{claim:pathinterp}
    For any $1\leq p\leq n-k$,
    \begin{align*}
        [s_{(1^{k})}]\sbar_{(k)}(x_1,\dots, x_n) = P_{p,p+k}
    \end{align*}
\end{claim}
\begin{proof}
    For each path counted by $P_{p,p+k}$, consider the sequence of $y$-coordinates at which the non-vertical steps in the path begin.  Because there are no down-steps, this sequence will be unique among all paths counted by $P_{a,b}$, and will be weakly increasing.  Additionally, after performing a non-vertical step starting at point $(x,y)$, the next non-vertical step can start at any $y$-coordinate strictly larger than $y$, or can start at $y$-coordinate $y$ as long as the last step performed was not a diagonal step \emph{i.e.}, $x\not=y$.  As a result, each path counted by $P_{p,p+k}$ corresponds to a weakly increasing sequence $1<c_1\leq \dots \leq c_{k}\leq n$ such that for all $i$,
    $c_{i}$ and $c_{i+1}$ do not both equal $p+i-1$.  If we fill the partition $(k)$ with the sequence  $c_1,\dots, c_{k}$, we get exactly the elements of $\{T\in {\rm SSYT}((k),n): T(1,1)>1, (T(1,i),T(1,i+1)\not= (p+i,p+i)\}$, which by Corollary~\ref{cor:rowcolumnmult} has $[s_{(1^{k})}]\sbar_{(k)}(x_1,\dots, x_n)$ elements, proving the claim.
\end{proof}

Using Claim~\ref{claim:pathinterp} and Equation~\ref{eq:det1}, we have that for any $\ell(\lambda)\leq p\leq n-\lambda_1$

\begin{align*}
\label{eq:det2}
    [s_{(1^{|\lambda|})}]\sbar_{\lambda}(x_1,\dots, x_n) = \det(P_{p+1-j,p+1+\lambda_i-i})_{1\leq i,j\leq \ell(\lambda)}.
\end{align*}

By Lindstr\"om-Gessel-Viennot (\cite{gessel}), and the fact that the graph is planar, this determinant is equal to the number of tuples of non-intersecting paths that connect each $(p+1-j,1)$ to $(p+1+\lambda_i-i,n)$ for $1\leq i \leq \ell(\lambda)$.  Consider one such tuple, and construct a filling $T$ of shape $\lambda$ such that $T(r,c)$ equals the $y$-coordinate of the $c$th non-vertical step of the path that goes from $(p+1-r,1)$ to $(p+1+\lambda_r-r,n)$.  From Claim~\ref{claim:pathinterp}, this means that the content of each row of $T$ will be a weakly increasing sequence, except that $T(r,c)$ and $T(r,c+1)$ cannot both be $p+c-r$.  Similarly, two adjacent paths will intersect if and only if there is some $c$ such that the the endpoint of $c$th non-vertical step of the left path is weakly below the start point of the $c$th non-vertical step of the right path.  As a result, the paths are non-intersecting if and only if for all $(r,c)\in \lambda$, either $T(r-1,c)<T(r,c)$ or $T(r-1,c)=T(r,c) = p+c-r$.  As a result, the number of tuples of non-intersecting lattice paths will be equal to the number of such fillings, completing the proof.

\end{proof}

\section{A representation-theoretic interpretation of the column case}

We begin by passing to infinitely many variables in order to view $\sbar_{(1^k)}$ as a symmetric \emph{function}, instead of a symmetric polynomial. In doing so, we allow the integers $n$ and $k$ to remain as parameters in our definition. We use the expression in Lemma \ref{lemma:alt} above, replacing $i$ with $k-i$ (in order to more clearly interpret the binomial coefficient appearing the definition). 

\begin{definition}
\label{def:c(x)}
Let $n \geq k$ be positive integers. Let $c_{n,k}(x) = c_{n,k}(x_1, x_2, \dots)$ be the symmetric function
\[
c_{n,k}(x) = \sum_{i=0}^k(-1)^{k-i}\binom{n-k+i}{i}s_{(1^{(k-i)})}s_{(1)}^i
\]
\end{definition}
Note that $c_{n,k}(x_1, \dots, x_n, 0, 0, \dots) = \sbar_{(1^k)}(x_1, \dots, x_n)$. We proceed by exhibiting an $\mathfrak{S}_k$-module whose Frobenius characteristic is exactly $c_{n,k}$ (recall Definition \ref{def:frob}).

The following discussion generalizes Theorem 2.4 in \cite{reinerwebb}. The appearance of the binomial coefficient in Definition \ref{def:c(x)} precludes the use of poset techniques in what follows.

Let $A$ and $M$ be finite nonempty sets of integers. Let $\Delta_i(A, M)$ denote the set of pairs $(w, S)$, where $w$ is an \emph{injective word} of length $i$ on the alphabet $A$, and $S$ is a nondecreasing sequence of length $i$ of elements of the set $M$. Here by injective word of length $i$ on the alphabet $A$, we mean a sequence of elements of $A$ where no element is repeated. We write such a sequence by simply concatenating its elements. Let $C_i(A, M)$ denote the $\C$-vector space with basis $\Delta_i(A, M)$. 

\begin{example}
    Let $A = [2]$ and $M = [2]$. Then the elements of $\Delta_2(A, M)$ are
\[
\{(12, (1,1)), (12, (1,2)), (12, (2,2)), (21, (1,1)), (21, (1,2)), (21, (2, 2))\}
\]
\end{example}
We define an $\mathfrak{S}_k$ action on $C_i([k], [n-k+1])$ by letting
    \[
    \sigma \cdot (w_1\dots w_i, S) = \text{sign}(\sigma)( \sigma(w_{1})\dots \sigma(w_{i}), S)
    \]
(so $\mathfrak{S}_k$ acts trivially on $S$). This action has a familiar-looking Frobenius characteristic:

\begin{proposition}
\label{prop:characters}
The Frobenius characteristic of the $\mathfrak{S}_k$-module $C_i([k], [n-k+1])$ is given by
\[
\text{ch}(C_i([k], [n-k+1])) =  \binom{n-k+i}{i}s_{(1^{(k-i)})}s_{(1)}^i
\]
\end{proposition}

\begin{proof} Let $\Gamma$ denote the collection of nondecreasing sequences of elements of the set $[n-k+1]$, and note that $|\Gamma| = \binom{n-k+i}{i}$. As $\mathfrak{S}_k$-modules, we have that
\[
C_i([k], [n-k+1]) \cong \bigoplus_{S \in \Gamma} (\text{sign}_{k-i} \otimes \C[\mathfrak{S}_i]){\mathord\uparrow}^{\mathfrak{S}_k}_{\mathfrak{S}_{k-i} \times \mathfrak{S}_i}
\]
where $\text{sign}_{k-i}$ is the sign representation of $\mathfrak{S}_{k-i}$, and $\C[\mathfrak{S}_i]$ is the regular representation of $\mathfrak{S}_i$. The Frobenius image of each summand of the right-hand side is the product 
\[
\text{ch}(\text{sign}_{k-i}) \text{ch}(\C[\mathfrak{S}_i]) = s_{(1^{(k-i)})}s_{(1)}^i
\]
Taking the direct sum indexed by $\Gamma$ gives the factor $\binom{n-k+i}{i}$.

\end{proof}

We proceed to define an algebraic chain complex and study its homology. For $1 \leq i \leq k$, let $\partial_i: C_i(A, M) \to C_{i-1}(A, M)$ be defined by 
\[
\partial_i(w_1 \dots w_i, (j_1, \dots, j_i)) = \sum_{l=1}^i (-1)^{l-1} (w_1\dots \widehat{w_l}\dots w_i, (j_1, \dots, \widehat{j_l}, \dots, j_i))
\]

Looking at the case $A = [k], M = [n-k+1]$, the maps $\partial_i$ are $\mathfrak{S}_k$-equivariant and for $1 \leq i \leq k-1$, $\partial_{i} \circ \partial_{i+1}=0$, so the modules $C_i([k], [n-k+1])$ form a chain complex of $\mathfrak{S}_k$-modules. For arbitrary nonempty sets $A, M$, we let $C_*(A, M)$ denote the corresponding chain complex of vector spaces.

The following lemma is an adaptation of Theorem 3 in \cite{gan}.
\begin{lemma} 
\label{lem:homconcentration}
Let $H_i = H_i(C_*(A, M))$ denote the homology of the chain complex $C_*(A, M)$ in degree $i$. Then $H_i = 0$ for $0 \leq i < k$.
\end{lemma}

\begin{proof} The proof is essentially the same as that of Theorem 3 in \cite{gan}: We use induction on $|A|$, and modify the induction hypothesis so that the set $M$ is not fixed. We define the same filtration and now have that 
\begin{align*}
    F_p&C_*(A, M)/F_{p-1}C_*(A, M)  \cong \\
    & \bigoplus_{(w_1\dots w_{p-1} \max(A), (j_1, \dots, j_i)) \in \Delta_p(A, M)} C_{*-p}(A \setminus \{w_1, \dots, w_{p-1}, \max(A)\}, M \setminus [j_p - 1]) 
\end{align*}
and we can apply the induction hypothesis to the right-hand side.

Continuing as in \cite{gan}, we have that the map
\[
(w_1\dots w_r, (j_1, \dots, j_r)) \mapsto (\max(A)w_1\dots w_r, (1, j_1, \dots, j_r)) 
\]
gives a null-homotopy for the inclusion map $C_*(A \setminus \max(A), M) \to C_*(A, M)$. So the edge map $E^1_{0, k-1} \to E^\infty_{0, k-1} \subset H_{k-1}(C_*(A, M))$ is zero, and hence $E^\infty_{0,k-1} = 0$. Here we have used the theory of spectral sequences in the same way as \cite{gan}.
\end{proof}
We are now able to state the main result of this section.
\begin{theorem} 
\label{thm:c(x)}
\[
c_{n,k}(x) = \sum_{i=0}^k (-1)^{k-i} \text{ch}(C_i) = \text{ch}(H_k)
\]
\end{theorem}

\begin{proof}
    The first equality is given by Definition \ref{def:c(x)}, and the second equality follows from Lemma \ref{lem:homconcentration} and the Hopf trace formula.
\end{proof}

Also, one can see that the Schur function expansion of $c_{n,k}$ is the same as the Schur polynomial expansion of $\sbar_{(1^k)}(x_1,\dots, x_n)$, so Theorem \ref{thm:c(x)} gives a representation-theoretic proof of the Schur positivity of $\sbar_{(1^k)}(x_1,\dots, x_n)$.

\section{Conclusion and open problems}

We conclude with some explicit open problems related to this work. We have found combinatorial expressions for $[s_{\mu}] \sbar_\lambda(x_1, \dots, x_n)$ for $\mu$ being a single column, or one or two rows, and for $\lambda$ being a single row or single column. This naturally leads to
\begin{problem}
    Find a combinatorial expression for $[s_{\mu}] \sbar_\lambda(x_1, \dots, x_n)$ for general $\mu$ and $\lambda$.
\end{problem}
There is also a natural approach to the representation theory of the case when $\lambda$ is a single row. Again let $A$ and $M$ denote finite nonempty sets of integers. Let $\Delta_i ' (A, M)$ be the set of pairs $(w, S)$ of injective words of length $i$ on the alphabet $A$ and \emph{increasing} sequences $S$ of length $i$ of elements of the set $M$. Let $C'_i(A, M)$ denote the $\mathbb{C}$-vector space with basis $\Delta'_i(A, M)$.

Now define an $\mathfrak{S}_k$ action on $C'_i([k], [n+k-1])$ by letting 
\[
\sigma \cdot (w_1 \dots w_k, S) = (\sigma(w_1) \dots \sigma(w_k), S)
\]
Note that we no longer have a ``sign'' term in this action.

By a similar argument as in Proposition \ref{prop:characters}, we have that the Frobenius characteristic of $C'_i([k], [n+k-1])$ is given by
\[
ch(C'_i([k], [n+k-1]) = \dbinom{n+k-1}{i} s_{(k-i)}s_{(1)}^i
\]
Finally, we define the symmetric function
\[
r_{n, k}(x_1, x_2, \dots) = \sum_{i=0}^k (-1)^{k-i} \dbinom{n+k-1}{i} s_{(k-i)}s_{(1)}^i
\]
Then using Lemma \ref{lemma:alt}, we see that $r_{n, k}(x_1, \dots x_n, 0, 0, \dots) = \sbar_{(k)}(x_1, \dots, x_n)$. And so the Schur function expansion of $r_{n, k}$ will be the same as the Schur polynomial expansion of $\sbar_{(k)}$ \emph{as long as} $n \geq k$ (since otherwise the Schur functions appearing in the former expansion may contain terms indexed by partitions of length greater than $n$). We can define a chain complex in the same way as above, with the exact same chain maps. Attempting to follow the same argument as in Section 6, we are led to the following
\begin{problem}
    Let $H_i$ denote the $i^{th}$ homology group in this chain complex. Prove that $H_i = 0$ for $i < k$, as long as $|M| \geq 2|A|-1$.
\end{problem}
We have some computational evidence for this conjecture, but were not able to use the same filtration and spectral sequence method to prove it. A proof of this conjecture would give a homological model for the row case.

\newpage

\printbibliography

\end{document}